\documentclass{amsart}

\usepackage{graphicx}
\usepackage{amsmath}
\usepackage{amssymb}
\usepackage{mathtools}
\usepackage[utf8]{inputenc}
\usepackage{amsthm}
\usepackage{lipsum}
\usepackage{xcolor}
\usepackage{hyperref}

\newtheorem{introthm}{Theorem}[section]
\newtheorem{introdef}[introthm]{Definition}
\newtheorem{introcor}[introthm]{Corollary}

\newtheorem{thm}{Theorem}[subsection]

\newtheorem{prop}[thm]{Proposition}
\newtheorem{cor}[thm]{Corollary}
\newtheorem{ex}[thm]{Example}
\newtheorem{rem}[thm]{Remark}

\definecolor{myteal}{HTML}{00797d}
\hypersetup{
    colorlinks=false,%
    citebordercolor=green,%
    filebordercolor=red,%
    linkbordercolor=blue,%
    urlbordercolor=red}
\newcommand{\BB}{\mathbb{B}}
\newcommand{\RR}{\mathbb{R}}
\newcommand{\CC}{\mathbb{C}}

\title[NON-EXISTENCE OF FREE BOUNDARY MINIMAL MÖBIUS BANDS IN $\BB^3$]{NON-EXISTENCE OF FREE BOUNDARY MINIMAL MÖBIUS BANDS IN THE UNIT THREE-BALL}
\author{Carlos Andrés Toro Cardona}
\date{}

\begin{document}

\begin{abstract}
We prove the impossibility of constructing free boundary minimal Möbius bands in the Euclidean ball $\mathbb{B}^3$. This answers in the negative a question proposed by I. Fernández, L. Hauswirth and P. Mira.
\end{abstract}

\maketitle

\section{Introduction}
\noindent This paper concerns minimal surfaces in $\RR^n$ that meet $\BB^n$ orthogonally, where
\begin{equation*}
    \mathbb{B}^n=\{(x_1,x_2,\ldots,x_n)\in \mathbb{R}^n| x_1^2+x_2^2+\ldots+x_n^2\leq 1\}\subset \mathbb{R}^n,
\end{equation*}
is the $n$-dimensional Euclidean unit ball.
\begin{introdef}\label{def fbbmi}
    Let $(\Sigma,h)$ be a compact orientable Riemannian surface with boundary. A branched minimal immersion of $(\Sigma,h)$ in $\RR^n$ is a non-constant map $X:(\Sigma,h) \rightarrow \mathbb{R}^n$ such that:
\begin{itemize}
    \item[(i)] $X\in C^0(\Sigma, \mathbb{R}^n)\cap C^2(\mathring{\Sigma}, \mathbb{R}^n)$,
    \item[(ii)] $X$ is harmonic, and
    \item[(iii)] $X$ is almost conformal.
    \end{itemize}
Moreover, if the map $X:(\Sigma,h) \rightarrow \mathbb{R}^n$ also satisfies
    \begin{itemize}
    \item[(iv)] $X(\partial \Sigma)\subset \partial \mathbb{B}^n$, and
    \item[(v)] $\frac{\partial X}{\partial \nu^{\Sigma}}\perp \partial \mathbb{B}^n$ along $\Sigma \cap \partial \mathbb{B}^n$,
\end{itemize}
where $\nu^{\Sigma}$ is the unit conormal vector of $\Sigma$, then we say $X$ is a branched free boundary minimal immersion in the unit ball $\BB^n$.
\end{introdef}
\noindent We remark that by regularity theory, the map $X$ extends smoothly up to the boundary \cite{JN1975}, so (v) is well-defined.\\\\
\noindent Let $X:(\Sigma,h)\rightarrow \RR^n$ be a branched minimal immersion. It is a well-known fact that we can always cover $\Sigma$ by an isothermic atlas \cite{CHERN1955}. A point $p\in \Sigma$ is called a \textit{branch point} if $|X_u|_{\mathbb{R}^3}(0)=|X_v|_{\mathbb{R}^3}(0)=0$, where $z=u+iv$ is an isothermic chart around the point $p$. In the absence of branch points, we say that the map $X$ is a minimal \textit{immersion}. Moreover, when the map $X$ is injective and free of branch points, we say it is a minimal \textit{embedding}. \\\\ The existence of free boundary minimal surfaces of given topological type has been the subject of intensive research lately and several constructions have been made. We begin the discussion with examples of free boundary minimal embeddings in the particular case of the unit ball $\BB^3$.\\\\ The simplest free boundary minimal surfaces are the totally geodesic equatorial disks, which are the intersections of planes passing through the origin with the unit ball. The only other rotationally symmetric free boundary minimal surface is the critical catenoid.\\\\ 
In a seminal paper \cite{FS16}, R. Schoen and A. Fraser discovered a close relationship between the existence problem of free boundary minimal surfaces and the Steklov eigenvalue problem. As a result they were able to provide embedded orientable examples with genus zero and any number of boundary components.\\\\ Other techniques, like min-max theory and gluing methods, have been used to construct embedded free boundary minimal surfaces of either large genus or with a large number of boundary components. On one hand, there are examples \cite{AFOLHA}, \cite{KAPOULEASZOU}, of both genus zero or one with a large number of boundary components. On the other hand there are also examples \cite{CSW22}, \cite{KAPOULEASLI}, \cite{KM20}, \cite{KW17}, \cite{KETOVERG0}, of free boundary minimal surfaces, of high genus with three and four boundary components.\\\\ It is also known by the work of A. Carlotto, G. Franz and M. Schulz \cite{Franz} that for one boundary component and arbitrary genus it is always possible to construct an embedded orientable free boundary minimal surface in the unit ball with exactly that topological type. In a very recent breakthrough, M. Karpukhin, R. Kusner, P. McGrath and D. Stern proposed a way to realize any compact orientable surface with boundary as an embedded free boundary minimal surface in the unit ball $\mathbb{B}^3$ \cite[Theorem 1.2]{KarpukhinStern}.\\\\
There are also immersed but not embedded free boundary minimal surfaces in $\mathbb{B}^3$. I. Fernández, L. Hauswirth and P. Mira in their work \cite{Isabel} construct free boundary minimal annuli immersed on the unit ball with a discrete group of symmetries. In particular, they are not isometric to the critical catenoid. Notice how this construction is relevant regarding the critical catenoid conjecture \cite[Open Question 5]{MLI}, which states that any free boundary minimal \textit{embedded} annuli in the unit ball is the critical catenoid.\\\\We also have the notion of a free boundary minimal non-orientable surface:
\begin{introdef}\label{def fbbmi non-orientable}
    Let $(M,h)$ be a compact non-orientable Riemannian surface with boundary. Denote by $\widehat{M}$ the orientable double cover of $M$ with the pull-back metric under the covering map $\pi: \widehat{M}\rightarrow M$. A branched minimal immersion of $(M,h)$ in $\mathbb{R}^n$ is a non-constant map $X: (M,h) \rightarrow \mathbb{R}^n$ whose lift $\widehat{X}=X\circ \pi: (\widehat{M},\pi^{*}h) \rightarrow \mathbb{R}^n$ satisfies properties (i), (ii) and (iii) in Definition \ref{def fbbmi}. Moreover, if the lift map $\widehat{X}$ also satisfies properties (iv), and (v), then we say $X$ is a free boundary branched minimal immersion of $(M,h)$ in $\mathbb{B}^n$.
\end{introdef} 
\noindent So far nobody has found an example of a non-orientable free boundary minimal surface in $\mathbb{B}^3$. By topological reasons, such a non-orientable surface would have to be immersed and not embedded. I. Fernández, L. Hauswirth and P. Mira in \cite{Isabel} proposed the following question:\\\\
\textbf{Open Problem.} \textit{Are there free boundary minimal Möbius bands in $\mathbb{B}^3$?}\\\\
We provide a negative answer to this question, even if we allow branched immersions:
\begin{introthm}\label{ppal theorem}
    There are no free boundary branched minimal Möbius bands in $\mathbb{B}^3$.
\end{introthm}
\noindent The argument is based on the study of the Hopf quadratic differential in the orientable double cover by the annulus. We show that in the free boundary case, it has an explicit form which is incompatible with the transformation law under the anti-holomorphic deck map which changes orientation.\\\\
We point out that there are branched minimal immersions of the open Möbius band in $\RR^3$. The Henneberg surface was the first one described in the literature \cite{Henneberg}, and it has exactly two branch points. An open Möbius band in $\mathbb{R}^3$ free of branch points and with finite total curvature $-6\pi$ was first discovered by W. Meeks \cite{Meeks}. Several generalizations have been made for both examples. D. Moya and J. Pérez constructed for every integer $m\geq 1$, an open Möbius band $H_m$ in $\RR^3$ with exactly $m+1$ branch points \cite{Joaquin}, recovering the classical Henneberg surface for $m=1$. On the other hand, M. Oliveira \cite{Oliveira} extended the work of W. Meeks by constructing complete, minimal Möbius bands free of branch points in $\RR^3$ with one end and finite total curvature $-2\pi m$ for any $m$ odd, $m\geq 3$. See \cite{Toru}, \cite{Mira} for related and further developments. Theorem \ref{ppal theorem} shows that it is not possible to intersect any of these open Möbius bands with a ball in $\RR^3$ in such a way that the intersection is orthogonal at every point.\\\\ It is also interesting to mention that R. Schoen and A. Fraser in \cite[Proposition 7.1]{FS16} found an embedded minimal Möbius band in $\mathbb{R}^4$ intersecting the unit ball $\BB^4$ orthogonally. Its induced metric achieves the supremum of the first non-zero normalized Steklov eigenvalue $\overline{\sigma_1}$ among all smooth metrics (see Section \ref{subsection application} for more details about the spectral problem). For higher normalized Steklov eigenvalues, the same maximization technique has been studied among the class of $\mathbb{S}^1$-invariant metrics on the Möbius band by A. Fraser and P. Sargent in \cite[Theorem 1.1] {SargentFraser}. For each $k\geq 1$, they construct a free boundary minimal immersion of a Möbius band in $\BB^4$, whose induced metric realizes the supremum of the $k$-normalized Steklov eigenvalue $\overline{\sigma_k}$ among all $\mathbb{S}^1$-invariant metrics.\\\\ The non-existence result of Theorem \ref{ppal theorem} has a consequence on the study of critical metrics for Steklov eigenvalues. In fact, by the work of A. Fraser and R. Schoen \cite[Proposition 2.4]{FraserSchoen2013}, we know that if there exists a maximizing metric of the $k$-normalized Steklov eigenvalue then there exists a free boundary minimal immersion by $k$-eigenfunctions in a unit ball of some dimension which is a lower bound on the multiplicity of the $k$-eigenspace. More generally, M. Karpukhin and A. Métras remarked that the result proven by R. Schoen and A. Fraser can be generalized to a wider notion of critical metrics \cite[Remark 2.3.2]{Karpukhin}. Therefore we can state the following consequence:
\begin{introcor}\label{introcor}
    If $h$ is a critical metric for the k-normalized Steklov eigenvalue $\overline{\sigma_k}$ on a Möbius band then the multiplicity of $\sigma_k$ is at least four.
\end{introcor}
\noindent The paper is organized as follows. In Section \ref{Section 2}, we focus on some general facts about free boundary minimal surfaces in the unit three-ball, giving special attention to the Hopf differential and its fundamental properties. We also review the Enneper-Weierstrass representation and discuss some examples of open Möbius bands in the Euclidean space. Then in Section \ref{Section 3}, we specialize on the topological type of an annulus and show that the Hopf differential of free boundary minimal annuli in the unit ball has a very specific form. In Section \ref{Section 4}, we analyze branched minimal immersions of Möbius bands in the Euclidean space and we find an explicit transformation law of the Hopf differential associated to the double cover by the annulus. We show that the specific form of the Hopf differential is incompatible with the transformation law under the deck transformation on the orientable double cover that changes orientation, which proves Theorem \ref{ppal theorem}. Finally in Section \ref{subsection application} we briefly introduce the Steklov eigenvalue problem and prove Corollary \ref{introcor}.\\\\ \textbf{Acknowledgements.} I express my gratitude to professor Lucas Ambrozio, for all the insightful conversations we had and for his guidance. I extend my gratitude to Alberto Cerezo, Henrique Nogueira and Mateus Spezia for fruitful discussions and valuable ideas. I also thank CAPES - Coordenação de Aperfeiçoamento de Pessoal de Nível Superior and FAPERJ - Fundação Carlos Chagas Filho de Amparo à Pesquisa do Estado do Rio de Janeiro for supporting this research project.

\section{Generalities about free boundary minimal surfaces}\label{Section 2}
\noindent In what follows $(\Sigma,h)$ will denote a compact orientable Riemannian surface with boundary $\partial \Sigma$, possibly empty. We denote by $\nu^{\Sigma}$ the outer unit conormal to $\Sigma$, which is a vector tangent to the surface, perpendicular to the boundary $\partial \Sigma$ and in the outer direction with respect to the interior $\mathring{\Sigma}$ of $\Sigma.$ A smooth chart $(U_{\alpha},\psi_{\alpha})$ on $\Sigma$, with $\psi_{\alpha}:U_{\alpha}\subset \mathbb{C}\rightarrow \Sigma$ and $z_{\alpha}=u+iv\in U_{\alpha}$, is called an \textit{isothermic chart} when $|\partial_u\psi_{\alpha}|_h=|\partial_v\psi_{\alpha}|_h$ and $\langle \partial_u \psi_{\alpha},\partial_v\psi_{\alpha}\rangle_h=0$. It is a well established theorem \cite{CHERN1955} the existence of an isothermic atlas of $\Sigma$. The transition functions $ \psi_{\alpha,\beta}=\psi_{\beta}^{-1}\circ \psi_{\alpha}$ of such an atlas are always holomorphic, inducing on $\Sigma$ the structure of a Riemann surface. For each isothermic chart $(U_{\alpha},z_{\alpha})$ on $\Sigma$ we define $\partial_{z_{\alpha}}=\frac{1}{2}(\partial_u-i\partial_v)$ and $\partial_{\Bar{z}_{\alpha}}=\frac{1}{2}(\partial_u+i\partial_v)$.\\\\
We clarify some aspects regarding Definition \ref{def fbbmi}. We say that the map $X:(\Sigma,h)\rightarrow \RR^n$ is harmonic, if each of its coordinate functions satisfy $\Delta_{\Sigma}x^{i}=0$ for all $1\leq i\leq n$. On the other hand, $X:(\Sigma,h) \rightarrow \mathbb{R}^n$ is said to be an almost conformal map if $|X_u|_{\mathbb{R}^n}=|X_v|_{\mathbb{R}^n}$ and $\langle X_u, X_v \rangle_{\mathbb{R}^n}=0$, for any isothermic chart on $\Sigma$.
\subsection{The Hopf differential} Let $(\Sigma,h)$ be an orientable Riemannian surface with boundary. Fix an isothermal atlas $(U_\alpha,\psi_{\alpha})$ with $z_{\alpha}\in U_{\alpha}$. \textit{A quadratic differential Q on $\Sigma$} is any collection of the form $\{\phi_{\alpha}(z_{\alpha})dz_{\alpha}\otimes dz_{\alpha}\}$, where $\phi_{\alpha}$ are complex valued functions $\phi_{\alpha}:U_{\alpha}\subset \mathbb{C}\rightarrow \mathbb{C}$ satisfying the following transformation rule under coordinate changes by transition functions $z_{\beta}=\psi_{\alpha,\beta}(z_{\alpha})$:
\begin{equation}\label{eq: defquaddifferential}
    \phi_{\beta}(z_{\beta})=\frac{1}{\psi'_{\alpha,\beta}(z_{\alpha})^2} \phi_{\alpha}(z_{\alpha}). \tag{2.1.1}
\end{equation}
Therefore a quadratic differential $Q$ is invariant under coordinate changes
\begin{equation}\label{eq: invariancequaddiff}
    \phi_{\beta}(z_{\beta})dz_{\beta}\otimes dz_{\beta}= \phi_{\alpha}(z_{\alpha})dz_{\alpha}\otimes dz_{\alpha}.\tag{2.1.2}
\end{equation}
A quadratic differential $Q$ is called \textit{holomorphic} if the corresponding functions $\phi_{\alpha}$ are holomorphic. Moreover, if the functions $\phi_{\alpha}$ are real when restricted to the boundary $\partial \Sigma$, then we say the \textit{quadratic differential Q is real on the boundary.}\\\\Let $X:(\Sigma,h) \rightarrow \mathbb{R}^3$ be a minimal immersion with $\Sigma$ orientable. The extrinsic geometry is described by the second fundamental form $II:\mathcal{X}(\Sigma) \times \mathcal{X}(\Sigma)\rightarrow \RR$:
\begin{equation}\label{eq: defsecondftalform}
    II(Y,Z)=\langle \nabla^{\mathbb{R}^3}_YZ, N_\Sigma\rangle_{\RR^3},\tag{2.1.3}
\end{equation}
where $Y,Z\in \mathcal{X}(\Sigma)$ and $N_\Sigma$ is the unit normal vector field to $\Sigma$. \textit{An umbilical point} $p\in \Sigma$, is a point where 
\begin{equation}
    II_p(v,w)=\lambda h_p(v,w),\notag
\end{equation}
 for some $\lambda \in \RR$, and for all $v,w\in T_p\Sigma$. In particular, this implies that at an umbilical point, all directions are principal with the same normal curvature. We say the immersion $X$ is \textit{totally geodesic} if the second fundamental form vanishes identically, i.e. every point of $\Sigma$ is umbilic with $\lambda=0$. In this case the immersion is contained on a plane \cite{MANFREDO}.\\\\For an isothermal atlas $(U_{\alpha},\psi_{\alpha})$ we can consider the complex valued functions $\phi_{\alpha}:U_{\alpha}\rightarrow \mathbb{C}$ given by $\phi_{\alpha}(z_{\alpha})=II(\partial_{z_{\alpha}}\psi_{\alpha},\partial_{z_{\alpha}}\psi_{\alpha})$, where $II$ was extended $\CC$-linearly. A straightforward computation shows that this collection of functions satisfies the transformation law \ref{eq: defquaddifferential}. Therefore we define \textit{the Hopf differential} as the quadratic differential $\{ II(\partial_{z_{\alpha}}\psi_{\alpha},\partial_{z_{\alpha}}\psi_{\alpha})dz_{\alpha}\otimes dz_{\alpha} \}$, or just $II(\partial_z,\partial_z)dz^2$ for short.\\\\ From the Hopf differential, we can recover the second fundamental form by taking its real part:
\begin{equation}\label{eq: Second fundamental form as real part of Hopf differential}
    II=2Re\{II(\partial_z,\partial_z)dz^2\}. \tag{2.1.4}
\end{equation}
The next following result is a well-known fact \cite{HildebrandtI} about the Hopf differential of constant mean curvature immersions, although we only state it for the minimal case. Its proof relies on the Codazzi equation for such immersions:
\begin{thm}\label{thm: Hopf diff is holomorphic}
    Let $X:(\Sigma,h)\rightarrow \mathbb{R}^3$ be a minimal immersion. Then the Hopf differential is a holomorphic quadratic differential.
\end{thm}
\noindent In the special case of free boundary minimal immersions we have an additional very important property:
\begin{thm}[{\textit{Cf. }\cite[Proof of Theorem 1]{Nitsche1985}}]\label{thm: Hopf diff is real on boundary}
    Let $X:(\Sigma,h)\rightarrow \mathbb{R}^3$ be a free boundary minimal immersion in $\mathbb{B}^3$. Then the Hopf differential is holomorphic and real on the boundary $\partial \Sigma$.
\end{thm}
\begin{proof}
By Theorem \ref{thm: Hopf diff is holomorphic}, it is enough to show that the Hopf differential is real on the boundary $\partial \Sigma.$ Let $z=x+iy$ be an isothermal chart in the neighborhood of a point $p\in \partial \Sigma$, in such a way that $\{y=0\}$ parametrizes the boundary points and $\{y>0\}$ parametrizes the interior points of $\Sigma$ around $p$. Thus the outer conormal vector $\nu^{\Sigma}$ is parallel to $-\partial_y$, and $\partial_x$ is tangent to the boundary $\partial \Sigma.$ The free boundary hypothesis is that the angle of intersection between the sphere $\partial \BB^3$ and $\Sigma$ along the boundary curve of intersection is constant and equal to $\frac{\pi}{2}$. Since any curve on the sphere is a line of curvature, the well-known Joachimsthal theorem \cite{MANFREDO} implies that the boundary curve $\partial \Sigma\subset \BB^3$ is also a line of curvature of $\Sigma.$ That is, $\partial_x$ is a principal direction of $\Sigma$, with principal curvature $\kappa_x:$
\begin{equation}
    \nabla^{\RR^3}_{\partial_x}N_{\Sigma}=\kappa_x \partial_x.\notag
\end{equation}
Therefore, along the boundary curve $\partial \Sigma$ near $p$:
\begin{gather}
    Im\left[ II(\partial_z,\partial_z)\right]=-\frac{1}{2}II(\partial_x,\partial_y)=-\frac{1}{2}\langle \nabla^{\RR^3}_{\partial_x}\partial_y,N_{\Sigma}\rangle_{\RR^3}\notag\\=\frac{1}{2}\langle \partial_y,\nabla^{\RR^3}_{\partial_x}N_{\Sigma}\rangle_{\RR^3}=\frac{k_x}{2}\langle \partial_y,\partial_x \rangle_{\RR^3}=0.\notag
\end{gather}
In conclusion, the Hopf differential is real on the boundary $\partial \Sigma$ and the proof is completed.
\end{proof}  
\subsection{Branch points.}
We now summarize some important properties of the behaviour of branched minimal immersions around branch points. The subset of all branch points of a branched minimal immersion $X:(\Sigma,h)\rightarrow \RR^3$ is denoted by $\Sigma_0$.
\begin{prop}[{\textit{Cf. }\cite[Proposition 3.1.2]{HildebrandtI}}]\label{prop: Branch points are discrete}
Let $X:(\Sigma,h) \rightarrow \mathbb{R}^3$ be a free boundary branched minimal immersion in $\mathbb{B}^3$. Then $\Sigma_0\subset \Sigma$ is discrete.
\end{prop}
\begin{proof}
 Since $X(\Sigma)$ meets the sphere $\partial\mathbb{B}^3$ orthogonally along $\partial \Sigma$, and the sphere is analytic, it is possible to extend the map on a larger $\Tilde{\Sigma}$ uniquely \cite[Theorem 2.8.2']{HildebrandtII}. Therefore, boundary branch points of a free boundary branched minimal immersion in $\mathbb{B}^3$ can be thought as interior branch points of a branched minimal immersion in $\RR^3$. It will then be enough to show that the latter branch points are isolated.\\\\
Now we follow the standard argument in \cite{HildebrandtI}. Take a simply-connected isothermic chart $(U,z)$ of $\Sigma$. By definition, $X(u,v)=(x_1(u,v),x_2(u,v),x_3(u,v))$ is harmonic on $U$. Therefore we can find its harmonic conjugate
\begin{equation*}
    X^{*}(u,v)= (x_1^{*}((u,v),x_2^{*}((u,v),x_3^{*}((u,v)),
\end{equation*}
which satisfies the Cauchy-Riemann equations:
    \begin{gather}
        X_u=X_v^{*},\notag\\ X_v= -X_u^{*}.\notag
    \end{gather}
Hence the three coordinates $\phi_j:U\rightarrow \CC$ of the function 
\begin{equation*}
    f(z)=(\phi_1,\phi_2,\phi_3)=X(z)+iX^{*}(z),
\end{equation*}
are holomorphic. Using the Cauchy-Riemann equations, we obtain
\begin{equation}
    f'(z)=(\phi'_1,\phi'_2,\phi'_3)=X_u-iX_v.\notag
\end{equation}
Extending the Euclidean inner product in a $\mathbb{C}$-linear way, and using conformality, we have
\begin{equation}
    |\phi'_1|^2+|\phi'_2|^2+|\phi'_3|^2=\langle f',\overline{f'} \rangle=2|X_u|^2_{\mathbb{R}^3}.\notag
\end{equation}
Thus, the branch points $p\in \Sigma_0\cap U$ are characterized as the simultaneous zeros of the three holomorphic non-identically zero functions $\phi'_j$ and the conclusion follows.
\end{proof}
\noindent The ideas discussed in the proof of the previous theorem allow us to show the following asymptotic expansion around the branch points:
\begin{prop}[{\cite[Proposition 3.2.1]{HildebrandtI}}]\label{prop: expansion of Xz around branch points}
    Let $X:(\Sigma,h) \rightarrow \mathbb{R}^3$ be a free boundary branched minimal immersion in $\mathbb{B}^3$. For any branch point $p\in \Sigma_0$, either on the boundary or on the interior, and for each isothermic chart $(U,z)$ centered at $p$, there exist $r>0$, an integer number $\nu \geq 1$ and a complex vector $A\in \mathbb{C}^3$ with $A\neq 0$ and $\langle A,A \rangle=0$ such that, for all $z\in B_r(0)\cap U$,
\begin{equation}\label{eq: expansion Xz around branch points}
X_z(z)=Az^{\nu}+O\left(|z|^{\nu+1}\right). \tag{2.2.1}
\end{equation}
\end{prop}
\noindent The next proposition asserts that the Hopf differential can be extended holomorphically to the branch points:
\begin{prop}[\textit{Cf. }{\cite[Proposition 3.2.1]{HildebrandtI}}]\label{prop: Hopf diff extends to branch points}
    Let $X:(\Sigma,h) \rightarrow \mathbb{R}^3$ be a free boundary branched minimal immersion in $\BB^3$. Then the Hopf differential extends holomorphically to the branch points. Moreover, the extension is exactly zero on the branch points.
\end{prop}
\begin{proof}
 We reproduce the proof of \cite{HildebrandtI} in order to show that the Hopf differential extends to zero on the branch points. Let $p\in \Sigma_0$ be a branch point, and consider an isothermal chart $(U,z)$ centered at $p$. From Proposition \ref{prop: expansion of Xz around branch points}, there exists $r>0$ such that \ref{eq: expansion Xz around branch points} holds on $B_r(0)\cap U$. If we write $A=\frac{\alpha-i\beta}{2}$ for $\alpha,\ \beta\in \mathbb{R}^3$, the conditions $A\neq 0$ and $\langle A,A\rangle = 0$ are equivalent to $|\alpha|^2=|\beta|^2\neq 0$ and $\langle \alpha,\beta \rangle_{\mathbb{R}^3}=0$.\\\\Since $X_z=\frac{X_u-iX_v}{2}$, we deduce that
\begin{gather}
    X_u(z)=Re(z^\nu) \alpha +Im(z^\nu) \beta +O(|z|^{\nu+1}) ,\notag\\
    X_v(z)=-Im(z^\nu)\alpha +Re(z^\nu) \beta +O(|z|^{\nu+1}),\notag
\end{gather}
for all  $z\in B_r(0)\cap U$. Therefore
\begin{equation}
    X_u \wedge X_v (z)= |z|^{2\nu}(\alpha \wedge \beta)+O(|z|^{2\nu+1}),\notag
\end{equation}
for all  $z\in B_r(0)\cap U$. Finally, when $z$ approaches $0$ we obtain:
\begin{equation}\label{eq: limit normal at branch}
    \lim_{z\rightarrow 0}N_{\Sigma}(z)=\lim_{z\rightarrow 0}\frac{X_u \wedge X_v (z)}{|X_u \wedge X_v (z)|}=\frac{\alpha \wedge \beta}{|\alpha \wedge \beta|}.\tag{2.2.2}
\end{equation}
Hence the function $II(\partial_z,\partial_z)=\langle X_{zz}(z),N_{\Sigma}(z)\rangle_{{\mathbb{R}^3}}$ is continuous at $z=0$, and therefore by the Riemann removable singularity theorem \cite[Theorem 3.1]{Stein}, it extends holomorphically to $z=0$. In conclusion, the Hopf differential extends holomorphically to the branch points.\\\\ Now we show the extension is exactly zero at the branch points. When the order of the branch point is strictly greater than one:
\begin{equation}
    X_{zz}(z)=\nu Az^{\nu-1}+O(|z|^\nu), \notag
\end{equation}
for all  $z\in B_r(0)\cap U$ and so $\lim_{z\rightarrow 0}\langle X_{zz}(z),N_{\Sigma}(z)\rangle_{{\mathbb{R}^3}}=0.$ If the order of the branch point is exactly one, then:
\begin{gather*}
    \lim_{z\rightarrow 0}\langle X_{zz}(z),N_{\Sigma}(z)\rangle_{{\mathbb{R}^3}}=\frac{1}{|\alpha\wedge \beta|}\langle A,\alpha\wedge \beta\rangle_{{\mathbb{R}^3}}\\=\frac{1}{2|\alpha\wedge \beta|}\langle \alpha,\alpha\wedge \beta\rangle_{{\mathbb{R}^3}}-\frac{i}{2|\alpha\wedge \beta|}\langle \beta,\alpha\wedge \beta\rangle_{{\mathbb{R}^3}}=0.
\end{gather*}
So in any case we check that the Hopf differential extends to zero across the branch points.
\end{proof}
\subsection{The Enneper-Weierstrass Representation.}

The Weierstrass representation formula is a useful tool to construct branched minimal immersions in the Euclidean space $\mathbb{R}^3$.
\begin{thm}[{\cite[Theorem 3.3.1]{HildebrandtI}}]\label{thm: Weierstrass rep}
    Let $\Omega\subset \CC$ be a simply connected domain. For every non-planar branched minimal immersion $X:\Omega\rightarrow \RR^3$, there exist a non-zero holomorphic function $f$ on $\Omega$ and a non-zero meromorphic function $g$ on $\Omega$ such that $fg^2$ is holomorphic on $\Omega$ and 
\begin{equation}\label{eq: formula Weierstrass}
X(z)=X(z_0)+Re\int_{z_0}^z\left[\frac{1}{2}f(1-g^2), \frac{i}{2}f(1+g^2),fg\right]d\zeta,\tag{2.3.1} 
\end{equation}
for all $z,z_0$ in $\Omega$. Conversely, two functions $f$ and $g$ on $\Omega$ as above define, by means of \ref{eq: formula Weierstrass}, a branched minimal immersion $X:\Omega\rightarrow \RR^3$.
\end{thm}
\noindent Let $X:(\Sigma,h) \rightarrow \mathbb{R}^3$ be a branched minimal immersion in $\RR^3$. Locally in a simply connected isothermic chart $(U,\psi)$, according with the previous theorem, there exists holomorphic data $(g,f)$, representing $X$ as in equation \ref{eq: formula Weierstrass}. The function $g$ describes the Gauss map $N_{\Sigma}$ of the immersion, by means of the stereographic projection $pr:\mathbb{S}^2 \setminus \{N\} \rightarrow \mathbb{C}$, where $N=(0,0,1)$ is the north pole of the sphere i.e. for all $z\in U$,
 \begin{equation*}
     N_\Sigma(X(z))=pr^{-1}\circ g(z).
 \end{equation*}
 All the geometry of the immersion, that is, both the first and second fundamental forms can be calculated in terms of the Weierstrass holomorphic data. The conformal factor can be calculated according to the following equation
\begin{equation}\label{eq: conformal factor}
    \left(X\circ \psi\right)^{*}(can_{\RR^3})=\frac{1}{4}\left(1+|g|^2\right)^2|f|^2|dz|^2. \tag{2.3.2}
\end{equation}
Therefore, the branch points of $X$ are precisely the simultaneous zeros of $f$ and $fg^2$. As for the second fundamental form $II$, it is enough to provide the equation for the Hopf differential
\begin{equation}\label{eq: Hopf diff with holomorphic data}
    II(\partial_z,\partial_z)dz^2=-\frac{g'(z)f(z)}{2}dz^2, \tag{2.3.3}
\end{equation}
because of \ref{eq: Second fundamental form as real part of Hopf differential}. See \cite[Section 3.3]{HildebrandtI} for a deduction of equations \ref{eq: conformal factor} and \ref{eq: Hopf diff with holomorphic data}.
\subsection{Examples} We describe the associated holomorphic data of the Henneberg surface \cite{Henneberg} and of the minimal open Möbius band in $\mathbb{R}^3$ of Meeks \cite{Meeks}. While the former has branch points, the latter is free of branch points. We calculate the first fundamental form and the Hopf quadratic differential for both examples and determine the location of the zeros of the Hopf differential.\\\\ The Weierstrass representation for non-orientable minimal surfaces covered by a planar domain, which is the case of the Möbius band, can be defined precisely through the Weierstrass representation on the orientable double cover. In both examples, this double cover is the punctured plane $\CC^{*}=\mathbb{C}\setminus \{0\}.$ Although the domain is not simply connected, there is a global version of Theorem \ref{thm: Weierstrass rep} for non-simply connected Riemann surfaces \cite[Theorem 1.16]{Barbosa}, where it is enough to require that the holomorphic data satisfies a period condition, as both examples do. The deck transformation is the anti-holomorphic involution without fixed points $T:\CC^{*}\rightarrow \CC^{*}$ given by $T(z)=-\frac{1}{\overline{z}}$. The open Möbius band is then obtained as the quotient of $\CC^{*}$ by the $\mathbb{Z}_2$ action generated by $T$.
\begin{ex}[{Henneberg's minimal surface}]\label{Henneberg's minimal surface}
\normalfont{The holomorphic data for the Henneberg surface is:
\begin{gather}
   f(z)=\frac{z^4-1}{z^4},\notag\\ g(z)=z.\notag
\end{gather}
Notice that $f$ and $g$ satisfy the required condition of Theorem \ref{thm: Weierstrass rep}. Hence, we can construct, through the Enneper-Weierstrass representation, the minimal immersion  $X:\CC^{*}\rightarrow \mathbb{R}^3$ as in equation \ref{eq: formula Weierstrass}
\begin{equation*}
    X(z)=Re \left[\frac{\left(1-z^2\right)^3}{6z^3},i\frac{\left(1+z^2\right)^3}{6z^3}, \frac{\left(1-z^2\right)^2}{2z^2}\right],
\end{equation*}
where we set $X(1)=0$. It is readily verified that $X(z)=X(-\frac{1}{\Bar{z}})$ and therefore it descends to a well-defined map from an open Möbius band into the Euclidean space $\RR^3$.\\\\
We can calculate the first fundamental form using equation \ref{eq: conformal factor}:
\begin{equation}\label{eq: conformal factor Henneberg}
    ds^2=\frac{(1+|z|^2)^2|z^4-1|^2}{4|z|^8}|dz|^2. \tag{2.4.1}
\end{equation}
The conformal factor vanishes at the points $z=\pm 1$ and $z=\pm i$. Notice that both $z=1$ and $z=-1$, on the orientable double cover, project to the same point on the Möbius band, and the same can be said about $z=\pm i$. Therefore the Henneberg minimal open Möbius band has only two branch points.
Using equation \ref{eq: Hopf diff with holomorphic data} we obtain the Hopf quadratic differential for Henneberg's example
\begin{equation}\label{Hopf differential Henneberg}
     II(\partial_z,\partial_z)dz^2=-\frac{z^4-1}{2z^4}dz^2. \tag{2.4.2}
\end{equation}\\
We observe that the Hopf differential vanishes exactly at the branch points, as it should be according to Proposition \ref{prop: Hopf diff extends to branch points}}.
\end{ex}
\begin{ex}[{Meek's minimal open Möbius band}]\label{Meeks example}
     \normalfont{The holomorphic data for this case is:
\begin{gather}
    f(z)=2i\frac{(z-1)^2}{z^4},\notag\\ g(z)=\frac{z^2(z+1)}{z-1}\notag. 
\end{gather}
Since the meromorphic function $g$ has a simple pole at $z=1$, which is a zero of the holomorphic function $f$ of order two, the function $fg^2$ is holomorphic on $\CC^{*}$. The three coordinate functions of the minimal immersion  $X=(X_1,X_2,X_3)$ associated to this holomorphic data are calculated by means of equation \ref{eq: formula Weierstrass}
\begin{gather*}
X_1(z)=Re\left[-\frac{i}{3}\left(z+\frac{1}{z}\right)^3-i\left(z^2-\frac{1}{z^2}\right)\right],\\X_2(z)=Re\left[-\frac{1}{3}\left(z^3-\frac{1}{z^3}\right)-\left(z-\frac{1}{z}\right)^2-\left(z-\frac{1}{z}\right)\right],\\X_3(z)=Re\left[2i\frac{1+z^2}{z^2}\right],
\end{gather*}
where we set $X(1)=0$. From these equations it can be checked that $X$ is invariant by the deck transformation, $X(z)=X(-\frac{1}{\overline{z}})$. It follows that the immersion $X$ induces a well-defined minimal immersion of an open Möbius band into $\mathbb{R}^3$. A straightforward computation using \ref{eq: conformal factor} gives
\begin{equation}\label{eq: conformal factor Meeks}
    ds^2=\frac{(|z-1|^2+|z|^4|z+1|^2)^2}{|z|^8}|dz|^2.
    \tag{2.4.3}
\end{equation}
Since the conformal factor never vanishes on $\CC^{*}$, the immersion is free of branch points. On the other hand, using equation \ref{eq: Hopf diff with holomorphic data} we obtain the Hopf quadratic differential
\begin{equation}\label{Hopf diff Meeks}
     II(\partial_z,\partial_z)dz^2=-2i\frac{z^2-z-1}{z^3}dz^2.
     \tag{2.4.4}
\end{equation}
The vanishing points of the Hopf differential on the double cover are $z=\frac{1\pm \sqrt{5}}{2}.$ Since these two points project to the same point on the Möbius band, we conclude it has exactly one umbilical point.}
\end{ex}

\section{Free boundary minimal annuli}\label{Section 3}
\subsection{Uniformization of annuli} In order to investigate the existence of all free boundary branched minimal immersions of Möbius bands in $\mathbb{B}^3$, it is useful to have a canonical way to describe its oriented double cover with the induced covering metric. For that, let us consider the annulus on the complex plane with its complex structure induced by the canonical Euclidean metric
\begin{equation}
    A_{a,b}=\{z\in \mathbb{C}|\ 0<a\leq |z| \leq b<+\infty \}\subset \mathbb{C}. \notag
\end{equation}
This annulus is always conformally equivalent to $A_{\frac{1}{R},R}$ with $R=\sqrt{\frac{b}{a}}$, because homothety preserves the conformal class. In fact, the moduli space of an annulus is one dimensional:
\begin{thm}[{\cite[Proposition 3.2.1]{Hubbard}}]\label{thm: uniformization of annuli}
Let $(N,h)$ be Riemannian surface diffeomorphic to the cylinder $\mathbb{S}^1\times [0,1]$. Then there exist $R>0$ and a conformal diffeomorphism $\phi:(N,h)\rightarrow (A_{\frac{1}{R},R},can)$.
\end{thm}
\noindent Also, we need the classification of the automorphisms of the annulus $(A_{a,b},can)$, that is, the group of all biholomorphisms from the annulus to itself:
\begin{prop}[{\cite[Example 6.1.13]{BasicComplex}}]\label{prop: automorphism group annulus}
    Let $\phi: (A_{a,b},can)\rightarrow (A_{a,b},can)$  be an automorphism. Then, there exists $\theta_0\in \mathbb{R}$ such that, for all $z\in A_{a,b}$, either $\phi(z)=e^{i\theta_0}z$ or $\phi(z)=e^{i\theta_0} \frac{ab}{z}$.
\end{prop}
\subsection{The Hopf differential for free boundary minimal annuli.} The goal in this section is to prove that any free boundary branched minimal annulus in $\BB^3$ has a very specific Hopf differential.
\begin{thm}[{\textit{Cf. }\cite[Section 4]{JLee}}]\label{thm: Hopf diff free boundary annulus}
    Let $X:(N,h) \rightarrow \mathbb{R}^3$ be a free boundary branched minimal immersion of an annulus in $\mathbb{B}^3$. Then its Hopf differential has the following form
\begin{equation}\label{eq: formula Hopf diff free boundary annulus}
    II(\partial_z,\partial_z)dz^2=\frac{C_0}{z^2}dz^2, \tag{3.2.1}
\end{equation}
where $C_0$ is a real constant.
\end{thm}
\begin{proof}
We merely adapt the proof of J. Lee and E. Yeon \cite{JLee} to the case of branched immersions. By the Uniformization Theorem \ref{thm: uniformization of annuli}, assume without loss of generality that the annulus is the canonical one $(A_{a,b},can)$. Additionally, by Proposition \ref{prop: Hopf diff extends to branch points}, the Hopf differential extends holomorphically across the branch points. Therefore we can define the holomorphic function on the annulus
\begin{equation}
f(z)=z^2II(\partial_z,\partial_z). \notag
\end{equation}
By Theorem \ref{thm: Hopf diff is real on boundary}, the free boundary condition implies that the imaginary part of the function $f$ vanishes on the boundary circles of the annulus. Since $Im(f)$ is harmonic on $A_{a,b}$ and vanishes on $\partial A_{a,b}$, the maximum principle implies $Im(f)\equiv 0$. Moreover, $Re(f)$ is the harmonic conjugate of $Im(f)$, so $Re(f)=C_0$ for some real constant $C_0$. Hence 
\begin{equation}
II(\partial_z,\partial_z)dz^2=\frac{C_0}{z^2} dz^2,\notag
\end{equation}
and the claim is proved.
\end{proof}
\noindent The explicit formula \ref{eq: formula Hopf diff free boundary annulus} has the following immediate consequence, depending on whether the constant $C_0$ is zero or not:
\begin{cor}\label{cor: annuli are free of umbilical points}
    Any free boundary branched minimal immersion of an annulus in the unit ball $\mathbb{B}^3$ is either totally geodesic or has neither branch points nor umbilical points.
\end{cor}
\begin{ex}[{A. Cerezo's example}]\label{ex: alberto example}
    \normalfont{There are examples of totally geodesic free boundary branched minimal immersions of annuli type in the unit ball $\BB^3$. Consider a horizontal strip $S=\{z=u+iv\in \CC| -1\leq v\leq 1\}$ and the holomorphic map $X:S\rightarrow \CC$ given by $X(z)=\sin(z)$ which can be written as
    \begin{equation*}
        X(u,v)=\sin(u)\cosh(v)+i\cos(u)\sinh(v).
    \end{equation*}
    For fixed values $v_0\in \RR$, the $u$-curves $X(u,v_0)$ are ellipses centered at the origin, and the map $X$ is $u$-periodic $X(u+2\pi,v_0)=X(u,v_0)$. Therefore it factors to a holomorphic map from an annulus $A$ into some planar ellipse which is conformally equivalent to the disk. By composing $X$ with an appropriate biholomorphism followed by the inclusion of the disk in the unit ball, we obtain a free boundary minimal immersion $\Tilde{X}:A\rightarrow \mathbb{B}^3$ which has exactly two branch points of order two. In fact, $X'(z)=\cos(z)$, so $X'(p)=0$ exactly at the points $p=(k+\frac{1}{2})\pi$ with $k\in \mathbb{Z}$}. By the $2\pi$-periodicity of $X$, the map $\Tilde{X}$ has only two branch points in the annulus.
\end{ex}
\section{Free boundary minimal Möbius bands}\label{Section 4}

\subsection{Uniformization of Möbius bands} In this section we consider (compact) Möbius bands. We start by finding the explicit form of the deck transformation on the double cover.
\begin{thm}\label{thm: Uniformization mobius bands}
    Let $(M,h)$ be a Möbius band. Then there exists a positive radius $R$ such that $(M,h)$ is isometric to the quotient of the canonical annulus $(A_{\frac{1}{R},R},can)$ by the $\mathbb{Z}_2$ action of the anti-holomorphic isometry $T(z)=-\frac{1}{\overline{z}}$.
\end{thm}

\begin{proof}
Let $\pi:\widehat{M}\rightarrow M$ be the orientable double cover of $M$. $\widehat{M}$ is diffemorphic to $\mathbb{S}^1\times [0,1]$. Endow $\widehat{M}$ with the pull-back metric $\widehat{h}=\pi^{*}h$ inducing a complex structure on $\widehat{M}$ since it is also orientable. Let $I$ be the deck transformation on the double cover $\widehat{M}$ that changes orientation, which is an involution $I\circ I=Id_{\widehat{M}}$, with no fixed points. Since $I:\widehat{M}\rightarrow \widehat{M}$ is an orientation-reversing isometry, it is an anti-holomorphic map on the Riemann complex structure of $\widehat{M}$. The Möbius band $M$ is then recovered isometrically as $(M,h)=(\widehat{M},\widehat{h})/\langle I \rangle$. By Theorem \ref{thm: uniformization of annuli}, there exists $R>0$ and a conformal diffeomorphism $\phi:(A_{\frac{1}{R},R},can)\rightarrow (\widehat{M},\widehat{h}).$\\\\ Consider the transformation $T:A_{\frac{1}{R},R}\rightarrow A_{\frac{1}{R},R}$ defined by conjugation of the involution $I$
\begin{equation}
    T=\phi^{-1}\circ I \circ \phi. \notag
\end{equation}
We check immediately that $T$ is an involution which is also anti-holomorphic and without fixed points. Therefore $\overline{T}$ is an automorphism of the annulus $A_{\frac{1}{R},R}$. By Theorem \ref{prop: automorphism group annulus}, there exists $\theta_{0}\in \mathbb{R}$ such that either $T(z)=e^{-i\theta_0}\overline{z}$ or that $T(z)=e^{-i\theta_0}\frac{1}{\overline{z}}$. The first possibility is ruled out, because it has infinitely many fixed points of the form $z_0=r_0e^{-\frac{i\theta_0}{2}}$, with $\frac{1}{R}\leq r_0 \leq R$. It follows that the transformation is $T(z)=e^{-i\theta_0}\frac{1}{\overline{z}}$ for some $\theta_0\in \RR$. Since $T^2=Id$, $e^{-2i\theta_0}=1$, so $\theta_0=\pi\cdot l$, where $l\in \mathbb{Z}$. Therefore the only possibilities for the anti-holomorphic involution $T$ are either $T(z)=\frac{1}{\overline{z}}$ or $T(z)=-\frac{1}{\overline{z}}$. Since we recover topologically the Möbius band as the quotient $A_{\frac{1}{R},R}/\langle T \rangle$, the only possibility is to have $T(z)=-\frac{1}{\overline{z}}$.
\end{proof}
\begin{rem}\label{rem: canonical way to see mobius bands}
    Theorem \ref{thm: Uniformization mobius bands} allows us to view any branched minimal immersion $X:(M,h) \rightarrow \mathbb{R}^3$ of a Möbius band in a canonical way. In fact, by the Uniformization Theorem \ref{thm: uniformization of annuli}, we can always see its lift to the double cover as the branched minimal immersion $\widehat{X}=X\circ \pi:(A_R,can)\rightarrow \RR^3$ of the canonical annulus $A_R:=A_{\frac{1}{R},R}$ for some radius $R>0$. Moreover, $\widehat{X}$ is such that $\widehat{X}(T(z))=\widehat{X}(z)$ for all $z\in A_R$, where $T(z)=-\frac{1}{\overline{z}}$ is the deck transformation.
\end{rem}
\begin{rem}\label{rem: Mobius bands cannot be on plane}
    No branched minimal immersion $X:(M,h) \rightarrow \mathbb{R}^3$ of a non-orientable surface has its image contained on a plane. Otherwise, the lifted map $\widehat{X}$ to the orientable double cover $\widehat{\Sigma}$, has a normal vector parallel to a fixed direction, in opposite directions at points in the same fiber of the projection map $\pi:\widehat{\Sigma}\rightarrow \Sigma$. Consequently, by continuity the normal vector is zero at some point, which is a contradiction because the normal vector is non-zero at regular points, and also non-zero at the branch points according to equation \ref{eq: limit normal at branch}.
\end{rem}
\subsection{The transformation law of the Hopf differential}
Given a branched minimal immersion of a Möbius band in $\RR^3$, we derive certain restrictions imposed on the Hopf quadratic differential of its double cover. Given $R>0$, for each $\theta\in \RR$ we denote the slit annulus
\begin{equation}
    U_{\theta}=A_R\setminus \{t\cdot e^{i\theta}| \frac{1}{R}\leq t \leq R \} \subset A_R, \notag
\end{equation}
which is a simply connected domain.
\begin{prop}[{\textit{Cf. }\cite[Proposition 1]{Meeks}}]\label{prop: transformation law of holomorphic data}
    Let $X:(M,h) \rightarrow \mathbb{R}^3$ be a branched minimal immersion of a Möbius band. Then, for each $\theta\in \mathbb{R}$, the associated holomorphic data $(g,f)$ on the slit annulus $U_{\theta}$ satisfies for every $z\in U_{\theta}\cap U_{-\theta}$ the following restrictions:
\begin{equation}\label{eq: transformation law of g function}
    g(T(z))=T(g(z)), \tag{4.2.1}
\end{equation}
\begin{equation}\label{eq: transformation law for f function}
    f(T(z))= -\overline{f(z)(zg(z))^2},\tag{4.2.2}
\end{equation}
where $T(z)=-\frac{1}{\overline{z}}$ is the deck transformation of the orientable double cover.
\end{prop}

\begin{proof} 
We include the proof of \cite{Meeks} for the sake of completeness.\\\\ By Remarks \ref{rem: canonical way to see mobius bands} and \ref{rem: Mobius bands cannot be on plane} there is an induced non-planar branched minimal immersion $\widehat{X}:(A_R,can)\rightarrow \RR^3$ on the orientable double cover with deck transformation $T(z)=-\frac{1}{\overline{z}}$. Fixing $\theta \in \RR$, by the Enneper-Weierstrass representation Theorem \ref{thm: Weierstrass rep} there exist a non-zero meromorphic function $g:U_{\theta}\rightarrow \mathbb{C}$ and a non-zero holomorphic function $f:U_{\theta}\rightarrow \mathbb{C}$ such that 
\begin{equation}\label{eq: Weierstrass formula for law transormation}
    \widehat{X}(z)=\widehat{X}(p_0)+Re\int_{p_0}^z\left[\frac{1}{2}f(1-g^2),\frac{i}{2}f(1+g^2),fg\right]d\zeta, \tag{4.2.3}
\end{equation}
holds for all $z$ and $p_0$ in $U_{\theta}$.\\\\ Fix $z_0$ in the doubly-connected set $U_{\theta}\cap U_{-\theta}$. We observe that $z_0$ and $T(z_0)$ will be in different components. Now take any $z\in U_{\theta}\cap U_{-\theta}$.  Since the deck transformation $T$ is orientation-reversing, the normal vectors of the surface at the points $\widehat{X}(z)$ and $\widehat{X}(T(z))$ are opposite to each other
\begin{equation}
    pr^{-1}\circ g(T(z))=N_{\Sigma}(\widehat{X}(T(z)))=-N_{\Sigma}(\widehat{X}(z))=-pr^{-1}\circ g(z). \notag
\end{equation}
Therefore we have that
\begin{equation*}
    \left\lbrace\frac{2g(T(z))}{1+|g(T(z))|^2},\frac{|g(T(z))|^2-1}{|g(T(z))|^2+1}\right\rbrace=-\left\lbrace\frac{2g(z)}{1+|g(z)|^2},\frac{|g(z)|^2-1}{|g(z)|^2+1}\right\rbrace.
\end{equation*}
From here we easily obtain the following two relations
\begin{gather}
    |g(z)|^2|g(T(z))|^2=1, \notag\\g(T(z))=-\frac{1+|g(T(z))|^2}{1+|g(z)|^2}g(z). \notag
\end{gather}
Hence
\begin{equation}
    g(T(z))=-\frac{g(z)}{|g(z)|^2}=-\frac{1}{\overline{g(z)}}=T(g(z)),\notag
\end{equation}
which proves \ref{eq: transformation law of g function}.\\\\ For the second part of the proof, let $p_0\in \{z_0,T(z_0)\}$, so that $z$ and $p_0$ are in the same connected component of $U_{\theta}\cap U_{-\theta}$. Using equation \ref{eq: Weierstrass formula for law transormation} and taking the third coordinate on the identity $\widehat{X}(T(z))=\widehat{X}(z)$, we have:
\begin{gather*}
    \widehat{X}_3(p_0)+Re\int_{p_0}^z f(\zeta)g(\zeta)d\zeta=\widehat{X}_3(z)=\widehat{X}_3(T(z))= \widehat{X}_3(p_0)+Re\int_{p_0}^{T(z)} f(\zeta)g(\zeta)d\zeta\\\\=\widehat{X}_3(p_0)+Re\int_{p_0}^{T(p_0)} f(\zeta)g(\zeta)d\zeta+Re\int_{T(p_0)}^{T(z)} f(\zeta)g(\zeta)d\zeta.
    \end{gather*}
Let us define the constant $v_0\in \RR$ by
\begin{equation*}
    v_0=Re\int_{p_0}^{T(p_0)} f(\zeta)g(\zeta)d\zeta,
\end{equation*}
then, we obtain
\begin{gather*}
    \widehat{X}_3(p_0)+Re\int_{p_0}^z f(\zeta)g(\zeta)d\zeta=\widehat{X}_3(p_0)+v_0+Re\int_{p_0}^{z} T^{*}\left(f(\zeta)g(\zeta)d\zeta\right)\\\\=\widehat{X}_3(p_0)+v_0+Re\int_{p_0}^{z} f(T(\zeta))g(T(\zeta))d\left(-\frac{1}{\overline{\zeta}}\right)\\\\=\widehat{X}_3(p_0)+v_0+Re\int_{p_0}^{z} f(T(\zeta))\left(-\frac{1}{\overline{g(\zeta)}}\right)\left(\frac{d\overline{\zeta}}{\overline{\zeta}^2}\right)=\widehat{X}_3(p_0)+v_0-Re\int_{p_0}^{z} \frac{f(T(\zeta))}{\overline{g(\zeta)\zeta^2}}d\overline{\zeta}\\\\=\widehat{X}_3(p_0)+v_0-Re\overline{\int_{p_0}^{z} \frac{f(T(\zeta))}{\overline{g(\zeta)\zeta^2}}d\overline{\zeta}}=\widehat{X}_3(p_0)+v_0-Re\int_{0}^{z} \frac{\overline{f(T(\zeta))}}{g(\zeta)\zeta^2}d\zeta.
\end{gather*}
That is
\begin{equation}
     Re\int_{p_0}^z f(\zeta)g(\zeta)d\zeta=v_0-Re\int_{p_0}^{z} \frac{\overline{f(T(\zeta))}}{g(\zeta)\zeta^2}d\zeta. \notag
\end{equation}
Taking derivatives with respect to $z$, we conclude that
\begin{equation}
   f(z)g(z)=-\frac{\overline{f(T(z))}}{g(z)z^2}, \notag
\end{equation}
from which equation \ref{eq: transformation law for f function} follows.
\end{proof}
\noindent We are ready to state the transformation law satisfied by the Hopf differential.
\begin{thm}\label{thm: transformation law of the Hopf differential}
    Let $X:(M,h) \rightarrow \mathbb{R}^3$ be a branched minimal immersion of a Möbius band. Then the Hopf differential associated to the orientable double cover transforms under the anti-holomorphic deck involution $T$ according to the rule
\begin{equation}\label{eq: law transformation for Hopf differential}
T^{*}\left(II(\partial_z,\partial_z)dz^2\right)=-\overline{\left(II(\partial_z,\partial_z)dz^2\right)}. \tag{4.2.4}
\end{equation}
\end{thm}

\begin{proof}
Let $II(\partial_z,\partial_z)dz^2$ be the Hopf differential associated to the branched minimal immersion $\widehat{X}:(A_R,can)\rightarrow \RR^3$ as in Remark \ref{rem: canonical way to see mobius bands}. Since the Hopf differential is a $2$-tensor, we need to verify the law of transformation point-wisely. Fix $\xi\in A_R$ and consider the slit annulus $U_{\theta}$, where $\theta\in \mathbb{R}$ is chosen appropriately so that $\xi\in U_{\theta}\cap U_{-\theta}$.\\\\ The branched immersion $\widehat{X}$ is non-planar because from Remark \ref{rem: Mobius bands cannot be on plane}, the image of $X$ is not contained on a plane. By Theorem \ref{thm: Weierstrass rep}, there exist non-zero holomorphic data $(g,f)$ on $U_{\theta}$ so that the branched minimal immersion $\widehat{X}$ is represented by the formula \ref{eq: formula Weierstrass}. On the hand, by \ref{eq: Hopf diff with holomorphic data}, the quadratic Hopf differential can be written in this isothermic chart $U_{\theta}$ in the following way: 
\begin{equation}
    II(\partial_z,\partial_z)dz^2=-\frac{g'(z)f(z)}{2}dz^2.\notag
\end{equation}
Therefore, by Proposition \ref{prop: transformation law of holomorphic data}, we compute at $\xi\in A_R$:
\begin{gather*}
T^{*}\left(II(\partial_z,\partial_z)dz^2\right)= \frac{1}{2}T^{*}\left(-g'(z)f(z)dz^2\right)= \frac{1}{2}T^{*}\left(-dg \otimes f(z)dz\right)\\\\=-\frac{1}{2}d(g\circ T)\otimes (f\circ T)(z) d T(z)=-\frac{1}{2}d(T\circ g)\otimes \left(-\overline{f(z)(zg(z))^2} d\left(-\frac{1}{\overline{z}}\right) \right)\\\\=-\frac{1}{2}d\left(-\frac{1}{\overline{g}}\right)\otimes \left(-\overline{f(z)(zg(z))^2} \frac{\overline{dz}}{\overline{z}^2} \right)=-\frac{1}{2}\left(\frac{1}{\overline{g}^2}\overline{dg}\right)\otimes \left(-\overline{f(z)(zg(z))^2} \frac{\overline{dz}}{\overline{z}^2} \right)\\\\=\frac{1}{2}\overline{dg\otimes f(z)dz}=\frac{1}{2}\overline{g'(z)f(z)dz^2}=-\overline{\left(II(\partial_z,\partial_z)dz^2\right)}.
\end{gather*}
We conclude that
\begin{equation}
T^{*}\left(II(\partial_z,\partial_z)dz^2\right)=-\overline{\left(II(\partial_z,\partial_z)dz^2\right)}, \notag
\end{equation}
as we wanted to prove.
\end{proof}
\begin{ex}
    \normalfont{We want to come back to Examples \ref{Henneberg's minimal surface} and \ref{Meeks example}, in order to test the transformation law of Theorem \ref{thm: transformation law of the Hopf differential} for the Hopf differential under the anti-holomorphic deck involution $T(z)=-\frac{1}{\overline{z}}$. For the Henneberg open Möbius band we verify:
\begin{gather*}
    T^{*}\left(-\frac{z^4-1}{2z^4}dz^2\right)=-\frac{1}{2}\frac{\left(-\frac{1}{\overline{z}}\right)^4-1}{\left(-\frac{1}{\overline{z}}\right)^4}d\left(-\frac{1}{\overline{z}}\right)\otimes d\left(-\frac{1}{\overline{z}}\right)\\\\=-\frac{1}{2}\left(\frac{1-\overline{z}^4}{\overline{z}^4}\right)\overline{dz \otimes dz}=\frac{1}{2}\overline{\left(\frac{z^4-1}{z^4}\right)dz \otimes dz}.
\end{gather*}
Hence
\begin{equation}
     T^{*}\left(-\frac{z^4-1}{2z^4}dz^2\right)=-\left(\overline{-\frac{z^4-1}{2z^4}dz^2}\right),\notag
\end{equation}
in perfect agreement with the transformation law \ref{eq: law transformation for Hopf differential}. In the same way, we can check that the Hopf differential associated to Meeks open Möbius band also transforms as in Theorem \ref{thm: transformation law of the Hopf differential}.}
\end{ex}
\subsection{The Main Theorem}
\noindent We are in the conditions to prove the main theorem:
\begin{thm}\label{thm: main theorem}
    Let $X:(M,h) \rightarrow \mathbb{R}^3$ be a branched minimal immersion of a Möbius band in the Euclidean space $\mathbb{R}^3$. Then $X$ is not free boundary in the unit ball $\mathbb{B}^3$.
\end{thm}
\begin{proof}
Assume by contradiction that $X$ is a free boundary branched minimal immersion in the unit ball $\BB^3$. By Remark \ref{rem: Mobius bands cannot be on plane} this immersion cannot be planar. Then, we have an induced free boundary non-totally geodesic branched minimal immersion $\widehat{X}:(A_R,can)\rightarrow \BB^3$ of the corresponding double cover by the annulus in the unit ball $\BB^3$. According to Theorem \ref{thm: Hopf diff free boundary annulus} and Corollary \ref{cor: annuli are free of umbilical points}, the map $\widehat{X}$ is free of branch points, free of umbilical points and moreover its Hopf differential has the form:
\begin{equation}
    II(\partial_z,\partial_z)dz^2=\frac{C_0}{z^2}dz^2, \notag
\end{equation}
where $C_0$ is a real non-zero constant. Under the anti-holomorphic deck transformation $T(z)=-\frac{1}{\overline{z}}$, the Hopf differential transforms in the following way
\begin{gather*}
T^{*}\left(II(\partial_z,\partial_z)dz^2\right)=T^{*}\left(\frac{C_0}{z^2}dz^2\right)=\frac{C_0}{(-\frac{1}{\overline{z}})^2}\left(d\left(-\frac{1}{\overline{z}}\right)\right)^2\notag\\\\=C_0\overline{z}^2\left(\frac{\overline{dz}}{\overline{z}^2}\right)^2=\overline{\frac{C_0}{z^2}dz^2}= \overline{\left(II(\partial_z,\partial_z)dz^2\right)}. \notag
\end{gather*}
However, from Theorem \ref{thm: transformation law of the Hopf differential}, we obtain
\begin{equation}
\overline{\left(II(\partial_z,\partial_z)dz^2\right)}=T^{*}\left(II(\partial_z,\partial_z)dz^2\right)=-\overline{\left(II(\partial_z,\partial_z)dz^2\right)}.\notag
\end{equation}
This implies that the Hopf differential has to vanish identically, making $C_0=0$, which is a contradiction. In conclusion there are no free boundary branched minimal Möbius bands in the unit ball $\mathbb{B}^3$.
\end{proof}
\section{Application to the Steklov problem}\label{subsection application}
\subsection{Application.} We provide a direct consequence of Theorem \ref{thm: main theorem} by determining a lower bound on the multiplicity of Steklov eigenvalues of critical metrics.\\\\ Let $(\Sigma,h)$ be a compact Riemannian surface with boundary. The induced Dirichlet to Neumann-map $\mathcal{D}_h:C^{\infty}(\partial\Sigma)\rightarrow C^{\infty}(\partial\Sigma)$ is defined for all $u\in C^{\infty}(\partial\Sigma)$ by
\begin{equation*}
    \mathcal{D}_hu=\frac{\partial \widehat{u}}{\partial \nu^{\Sigma}},
\end{equation*}
where $\widehat{u}$ is the harmonic extension of $u$ to $\Sigma$ i.e. $\Delta_{\Sigma} \widehat{u}=0$. The operator $\mathcal{D}_h$ is self-adjoint, pseudodifferential and has discrete spectrum
\begin{equation*}
    0=\sigma_0(\Sigma,h)< \sigma_1(\Sigma,h)\leq \sigma_2(\Sigma,h)\leq \ldots \leq \sigma_k(\Sigma,h) \leq \ldots
\end{equation*}
The normalized $k$-Steklov eigenvalue $\overline{\sigma_k}$ is defined by 
\begin{equation*}
   \overline{\sigma_k}(\Sigma,h)=\sigma_k(\Sigma,h) \text{Length($\partial \Sigma,h)$}.
   \end{equation*}
\noindent Based on the definition of M. Karpukhin and A. Métras \cite[Definition 3.1]{Karpukhin}, we say that a metric $h$ is critical for the $k$-normalized Steklov eigenvalue $\overline{\sigma_k}$ if for all one-parameter smooth family of metrics $h(t)$ on $\Sigma$ with $h(0)=h$, we have either
\begin{equation*}
    \overline{\sigma_k}(\Sigma,h(t))\leq \overline{\sigma_k}(\Sigma,h)+o(t)\hspace{0.3cm}\text{or} \hspace{0.3cm} \overline{\sigma_k}(\Sigma,h)+o(t)\leq \overline{\sigma_k}(\Sigma,h(t)),
\end{equation*} as $t\rightarrow 0$. In this context we can formulate the following
\begin{thm}[{ \cite[Proposition 2.4]{FraserSchoen2013}, \cite[Remark 2.3.2]{Karpukhin}}]\label{thm: extremal metrics induce free boundary surfaces}
    If $\Sigma$ is a compact surface with boundary, and $h$ is a critical metric for the $k$-normalized Steklov eigenvalue $\overline{\sigma_k}$ in $\Sigma$, then there exist independent $k$-eigenfunctions $u_1,\ldots, u_n$ which give a free boundary minimal immersion $u=(u_1,\ldots,u_n):(\Sigma,h)\rightarrow \BB^n$.
\end{thm}
\begin{cor}
    If $h$ is a critical metric for the $k$-normalized Steklov eigenvalue $\overline{\sigma_k}$ on a Möbius band then the multiplicity of $\sigma_k$ is at least four.
\end{cor}

\begin{proof}
By Theorem \ref{thm: extremal metrics induce free boundary surfaces} there exists a free boundary minimal immersion of a Möbius band with metric $h$ by $k$-eigenfunctions in some unit ball. This unit ball cannot be the interval $[0,1]$ by dimensional reasons, nor the disk $\mathbb{D}$ by Remark \ref{rem: Mobius bands cannot be on plane} and nor the three ball $\BB^3$ by Theorem \ref{thm: main theorem}. Therefore the dimension of the ball should be greater or equal than four, and hence the multiplicity of the $k$-Steklov eigenspace must be at least four.
\end{proof}

\textsc{C. Toro: Impa, Rio de Janeiro RJ 22460-320 Brazil}\\\textit{Email address:} {\tt \textbf{carlos.toroc@impa.br}}
\end{document}